\documentclass[12pt,reqno]{amsart}

\usepackage[all]{xy}
\usepackage{amsthm,array,amssymb,amscd,amsfonts,latexsym}

\usepackage{enumerate}

\theoremstyle{plain}

\newtheorem{theorem}{Theorem}

\newtheorem{remark}{\bf Remark}
\newtheorem{corollary}{Corollary}

\theoremstyle{definition}

\newtheorem{nothing*}[theorem]{}
\newtheorem{subnothing*}[sub]{}
\newtheorem{example}{Example}

\theoremstyle{remark}

\begin{document}

\title[On algebraic group varieties]{On
algebraic group varieties}
\author[Vladimir  L. Popov]{Vladimir  L. Popov}
\address{Steklov Mathematical Institute,
Russian Academy of Sciences, Gub\-kina 8,
Moscow 119991, Russia}
\email{popovvl@mi-ras.ru}

\begin{abstract}
Several  results on presenting
an affine algebraic group va\-rie\-ty as a product of algebraic va\-rie\-ties are obtained.
 \end{abstract}

\maketitle

This note explores possibility of presenting
an affine algebraic group va\-rie\-ty as a product of algebraic varieties. As starting
point served the question of B.\;Ku\-nyavsky \cite{6} about the validity of the statement formulated below as Corollary of Theorem \ref{thm1}. For some special presen\-ta\-tions, their existence is proved in Theorem \ref{thm1}, and, on the contrary, nonexistence in Theorems \ref{thm2}--\ref{thm6}.

Let $G$ be a connected reductive algebraic group over an algebraically closed field $k$. The derived subgroup $D$ and the connected component $Z$ of the identity element of the center of the group $G$ are respectively
a connected semisimple algebraic group
and a torus (see \cite[Sect.\;14.2, Prop.\;(2)]{3}). The algebraic groups
$D\times Z$ and $G$ are not always isomor\-phic; the latter is equivalent to the equality
$D\cap Z=1$, which, in turn, is equivalent to the property that the isogeny of algebraic groups
$D\times Z\to G,\; (d, z)\mapsto d  z$, is their isomorphism.

\begin{theorem}\label{thm1}
There is an injective algebraic group homomorphism
$$\iota\colon Z\hookrightarrow G$$
such that
$\varphi\colon D\times Z\to G,\; (d, z)\mapsto d  \,\iota(z),$
is an isomorphism of al\-geb\-raic varieties.
\end{theorem}
\begin{corollary}
\label{iv} The
underlying varieties of
 \textup(generally noniso\-mor\-phic\textup)
 al\-geb\-raic groups $D\times Z$ and $G$ are isomorphic.
\end{corollary}

\begin{remark}{\rm
The proof of Theorem \ref{thm1} contains more information than its statement (the existence of $ \iota $ is proved by an explicit construction).}
\end{remark}


\begin{example}[{\rm \cite[Thm.\;8,\;Proof]{9}}]  Let he group $G$ be ${\rm GL}_n$. Then $D={\rm SL}_n$, $Z=\{{\rm diag}(t,\ldots, t)\mid t\in k^\times\}$, and  one can take
${\rm diag}(t,\ldots, t)\mapsto {\rm diag}(t, 1,\ldots, 1)$ as $\iota$. In this Example,  $G$ and $D \times Z$ are nonisomorphic algebraic groups.
\end{example}

\begin{proof}[Proof of Theorem {\rm\ref{thm1}}]
Let $T_D$ be a maximal torus of the group  $D$, and let $T_G$ be a maximal torus of the group $G$ containing $T_D$. The torus $T_D$ is a direct factor of the group $T_G$: in the latter, there is a torus $S$ such that the map
$
T_D\times S\to
T_G,\; (t, s)\mapsto ts,
$
is an isomorphism of algebraic groups (see
\cite[8.5, Cor.]{3}).
We shall show that
\begin{equation}\label{psi}
\psi\colon D\times S\to G, \quad (d,s)\mapsto d  s,
\end{equation}
is an isomorphism of algebraic varieties.

As is known
(see \cite[Sect. 14.2, Prop. (1),  (3)]{3}),
\begin{equation}\label{pro}
\mbox{\rm (a)} \,\;Z\subseteq T_G,\qquad
\mbox{\rm (b)}\,\; DZ=G.
\end{equation}

Let $g\in G$. In view of \eqref{pro}{\rm(b)}, we have $g=dz$ for some $d\in D$, $z\in Z$, and in view of  \eqref{pro}(a) and the definition of
 $S$, there are $t\in T_D$, $s\in S$ such that $z=ts$. We have $dt\in D$ and
$\psi(dt, s)=dts=g$. Therefore, the morphism $\psi$ is surjective.

Consider in $G$ a pair of mutually opposite Borel subgroups containing $T_G$.
The unipotent radicals $U$ and $U^-$ of these Borel subgroups lie in $D$.
Let $N_D(T_D)$ and $N_G(T_G)$ be the normalizers of tori $T_D$ and $T_G$ in the groups $D$ and $G$ respectively. Then $N_D(T_D)\subseteq N_G(T_G)$ in view of \eqref{pro}{\rm(b)}. The homomorphism  $N_D/T_D\to N_G/T_G$  induced by this embedding is an isomorphism of groups  (see \cite[IV.13]{3}), by which we identify them and denote by $W$.
For each $\sigma\in W$,  fix a representative $n_\sigma\in N_D(T_D)$. The group $U\cap n_{\sigma}U^-n_{\sigma}^{-1}$ does not depend on the choice of this representative,
since $T_D$ normalizes $U^-$; we denote it by $U'_{\sigma}$.

It follows from the Bruhat decomposition that for each $g\in G$,
there are uniquely defined $\sigma\in W$, $u\in U$, $u'\in U'_\sigma$ and
$t^{\ }_G\in T_G$
such that
$g=u'n_\sigma u t^{\ }_G$ (see \cite[Sect. 28.4, Thm.]{5}).  In view of the definition of $S$, there are uniquely defined $t^{\ }_D\in T_D$ and $s\in S$ such that $t^{\ }_G=t^{\ }_Ds$, and in view of
 $u', n_\sigma, u, t^{\ }_D\in  D$, the condition $g\in D$ is equivalent to the condition $s=1$. It follows from this and the definition of the morphism $\psi$ that the latter is injective.

Thus $\psi$ is a bijective morphism. Therefore, to prove that it is an isomor\-phism of algebraic varieties, it remains to prove its separability  (see  \cite[Sect. 18.2, Thm.]{3}). We have
${\rm Lie}\,G={\rm Lie}\,D+{\rm Lie}\,T_G$ (see \cite[Sect. 13.18, Thm.]{3}) and  ${\rm Lie}\,T_G={\rm Lie}\, T_D+{\rm Lie}\,S$ (in view of the definition of $S$). Therefore,
\begin{equation}\label{Lie}
{\rm Lie}\,G={\rm Lie}\,D+{\rm Lie}\,S.
\end{equation}
On the other hand, it is obvious from \eqref{psi} that the restrictions of the morphism $\psi$ to the subgroups $D\times\{1\}$ and $\{1\}\times S$
in $D\times S$ are isomor\-phisms respectively with subgroups $D$ and $S$ in $G$. Since
$
{\rm Lie}\,(D\times S)=
{\rm Lie}\,(D\times \{1\})
+ {\rm Lie}\,(\{1\} \times S),
$
it follows from \eqref{Lie} that the differential  of morphism  $\psi$ at the point $(1,1)$ is surjective.
Therefore (see   \cite[Sect. 17.3, Thm.]{3}), the morphism $\psi$ is separable.

Since $\psi$ is an isomorphism, it follows from \eqref{psi} that $\dim G=\dim D+\dim S$.
On the other hand,
\eqref{pro}{\rm(b)} and finiteness of $D\cap Z$ imply that $\dim G=\dim D+\dim Z$. Therefore,
$Z$ and $S$ and equidimensional, and hence isomorphic tori. Whence, as $\iota$ we can take the composition of any isomorphism of tori $Z\to S$ with the identity embedding
$S \hookrightarrow G$.
\end{proof}

 \begin{theorem}\label{thm2} An algebraic variety on which there is a nonconstant in\-ver\-tible regular function, cannot be a direct factor of a connected semisimple algebraic group variety.
\end{theorem}

\begin{proof}[Proof of Theorem {\rm\ref{thm2}}]
If the statement of Theorem \ref{thm2}  were not true, then the existence
the nonconstant invertible function specified in it would imply the existence of such a function
$f$ on a connected semisimple algebraic group.
Then, according to
\cite[Thm. 3]{10}, the function $f/f(1)$ would be a nontrivial character of this group, despite the fact that connected semi\-simple groups have no nontrivial characters.
\end{proof}

In Theorems \ref{thm3},
\ref{thm6} below we assume that $k=\mathbb C$; according to the Lefschetz principle, then they are valid for fields $k$ of characteristic zero. Below, topological terms refer to the Hausdorff $\mathbb C$-topology, homology and coho\-mo\-logy are singular, and the notation
$P\simeq Q $ means that the groups $P$ and $Q$ are isomorphic.

\begin{theorem}\label{thm3}
If a $d$-dimensional algebraic variety
$X$ is a direct factor of
a connected reductive algebraic group variety, then
$H_d(X, \mathbb Z)\simeq\mathbb Z$ and $H_i(X, \mathbb Z)=0$ for $i>d$.
\end{theorem}

\begin{proof}
Suppose that there are a connected algebraic reductive group $R$ and an algebraic variety
$Y$ such that the algebraic variety $R$ is isomorphic to $X\times Y$. Let $n:=\dim R$; then $\dim Y=n-d$. The algebraic varieties $X$ and $Y$ are irreducible, smooth, and affine. Therefore (see \cite[Thm. 7.1]{7}),
\begin{equation}\label{H}
H_i(X, \mathbb Z)=0 \;\;\mbox{for $i>d$}, \;\;\;H_j(Y, \mathbb Z)=0\;\;\mbox{for $j>n-d$}.
\end{equation}
By the universal coefficient theorem, for any algebraic variety $V$ and every $i$, we have
\begin{equation}\label{HHH}
H_i(V, \mathbb Q)\simeq H_i(V, \mathbb Z)\otimes \mathbb Q,
\end{equation}
and by the  K\"unneth formula,
\begin{equation}\label{HHHH}
H_n(R,\mathbb Q)\simeq H_n(X\times Y, \mathbb Q)\simeq \textstyle\bigoplus_{i+j=n}H_i(X, \mathbb Q)\otimes H_j(Y, \mathbb Q).
\end{equation}
Therefore, it follows from \eqref{H} that
\begin{equation}\label{HH}
H_n(R, \mathbb Q)\simeq H_d(X, \mathbb Q)\otimes H_{n-d}(Y, \mathbb Q).
\end{equation}

On the other hand, if $K$ is a maximal compact subgroup of the real Lie group  $R$, then the Iwasawa decomposition shows that $R$, as a topological manifold, is a product of $K$ and a Euclidean space, and therefore, the manifolds $R$ and $K$ have the same homology. Since the algebraic group $R$ is the complexification of  the real Lie group $K$, the dimension of the latter is $n$. Therefore, $H_n(K, \mathbb Q)\simeq\mathbb Q$ because $K$ is a closed connected orientable topological manifold. Whence, $H_n(R, \mathbb Q)\simeq\mathbb Q$. This and
\eqref{HH} imply that $H_d(X, \mathbb Q)\simeq \mathbb Q$. In turn, in view of \eqref{HHH}, this implies that
$H_d(X, \mathbb Z)\simeq \mathbb Z$ because $H_d(X, \mathbb Z)$ is a finitely generated   (see \cite[Sect.\;1.3]{4}),  torsion free (see \cite[Thm.\;1]{1}) Abelian group.
\end{proof}

\begin{corollary}
A contractible algebraic variety \textup (in particular, $\mathbb A^d $\textup) of positive dimension cannot be a direct factor of a
connected reduc\-tive algebraic group variety.
\end{corollary}

\begin{theorem}\label{thm5}
An algebraic curve cannot be a direct factor of a
connected semisimple algebraic group variety.
\end{theorem}

\begin{proof} Suppose an algebraic curve $X$ is a direct factor
  a connected semi\-simple algebraic group $R$ variety.
Then $X$ is irreducible, smooth, affine, and there is a surjective morphism
  $\pi\colon R\to X$.  In view of ra\-tio\-nality of the algebraic variety $R$
 (see \cite[14.14]{2}), the existence of  $\pi$ implies unira\-tionality, and therefore, by L\"uroth's
 theorem, rationality $X$.\;Hence $X$ is isomorphic to an open subset $U$ of $\mathbb A^1$.\;The case $U= \mathbb A^1$ is impossible due to Theorem \ref{thm3}. If $U\neq \mathbb A^1$, then on $X$ there is a nonconstant invertible regular function, which is impossible due to Theo\-rem \ref{thm2}.
\end{proof}

\begin{theorem}\label{thm6}
An algebraic surface cannot be a direct factor of
a con\-nected semisimple algebraic group variety.
\end{theorem}
\begin{proof} Suppose there are a connected semisimple algebraic group
  $R$ and the algebraic varieties  $X$ and $Y$ such that
$X$ is a surface and
$X\times Y$ is isomorphic to the algebraic variety
 $R$. We keep the notation of the proof of Theorem \ref{thm3}. Since $R$ is semisimple,  $K$ is semisimple as well.
Therefore, $H^1(K, \mathbb Q)=H^2(K, \mathbb Q)=0$ (see \cite[\S9, Thm. 4, Cor. 1]{8}).
Since $R$ and $K$ have the same homology, and $\mathbb Q$-vector spaces  $H^i(K, \mathbb Q)$ and $H_i(K, \mathbb Q)$ are dual to each other,
this yealds
\begin{equation}\label{zero}
H_1(R, \mathbb Q)=H_2(R, \mathbb Q)=0.
\end{equation}
Since $R$ is connected, $X$ and $Y$ are also connected. Therefore,
\begin{equation}\label{0}
H_0(X,\mathbb Q)=H_0(Y, \mathbb Q)=\mathbb Q.
\end{equation}
It follows from \eqref{HHHH}, \eqref{zero}, and \eqref{0} that $H_2(X,\mathbb Q)=0$. In view of
\eqref{HHH}, this contradicts Theorem \ref{thm3}, which completes the proof.
 \end{proof}

\begin{remark}{\rm
It seems plausible that,
using, in the spirit of  \cite{2}, \'etale coho\-mo\-logy in place of singular homology and cohomology,
one can adapt the proofs of Theorems \ref{thm3} and \ref{thm6} to the case of field $k$ of any characteristic.}
\end{remark}

\end{document}